\documentclass[11pt,draft,reqno]{amsart}

\usepackage[all]{xy}
\usepackage{amsthm,array,amssymb,amscd,amsfonts,latexsym}
\usepackage{enumerate}

{\catcode`\@=11
\gdef\n@te#1#2{\leavevmode\vadjust{%
 {\setbox\z@\hbox to\z@{\strut#1}%
  \setbox\z@\hbox{\raise\dp\strutbox\box\z@}\ht\z@=\z@\dp\z@=\z@%
  #2\box\z@}}}
\gdef\leftnote#1{\n@te{\hss#1\quad}{}}
\gdef\rightnote#1{\n@te{\quad\kern-\leftskip#1\hss}{\moveright\hsize}}
\gdef\?{\FN@\qumark}
\gdef\qumark{\ifx\next"\DN@"##1"{\leftnote{\rm##1}}\else
 \DN@{\leftnote{\rm??}}\fi{\rm??}\next@}}
\DeclareFontFamily{OT1}{wncyr}{\hyphenchar\font45
}
\DeclareFontShape{OT1}{wncyr}{m}{n}{%
   <5> <6> <7> <8> <9> gen * wncyr
   <10> <10.95> <12> <14.4> <17.28> <20.74>  <24.88>wncyr10}{}
\DeclareFontShape{OT1}{wncyr}{m}{it}{%
   <5> <6> <7> <8> <9> gen * wncyi
   <10> <10.95> <12> <14.4> <17.28> <20.74> <24.88> wncyi10}{}
\DeclareFontShape{OT1}{wncyr}{m}{sc}{%
   <5> <6> <7> <8> <9> <10> <10.95> <12> <14.4>
   <17.28> <20.74> <24.88>wncysc10}{}
\DeclareFontShape{OT1}{wncyr}{b}{n}{%
   <5> <6> <7> <8> <9> gen * wncyb
   <10> <10.95> <12> <14.4> <17.28> <20.74> <24.88>wncyb10}{}
\input cyracc.def

\DeclareMathSizes{9}{9}{7}{5} % This only is different.

\theoremstyle{plain}

\newtheorem*{thmnonumber}{Theorem}

\theoremstyle{definition}

\newtheorem{nothing*}[theorem]{}
\newtheorem{subnothing*}[sub]{}

\theoremstyle{remark}

\begin{document}
%%\renewcommand{\baselinestretch}{2}

%%\

%%\vskip -10mm

%%\

\title[On conjugacy of stabilizers of reductive group actions]
{On conjugacy of stabilizers of\\ reductive group actions}
\author[Vladimir  L. Popov]{Vladimir  L. Popov}
\address{Steklov Mathematical Institute,
Russian Academy of Sciences, Gubkina 8,
Mos\-cow\\ 119991, Russia}
\email{popovvl@mi-ras.ru}

%%\thanks{
%%}

%%\date{January 27, 2019}

%%\subjclass[2000]{}

%%\keywords{Affine cone, algebraic monoid, action}

\maketitle

\begin{abstract}
It is shown that the main result of
%%\cite{W}
N.\;R.\;Wallach, {\it Principal orbit type theorems for reductive
algebraic group actions and the Kempf--Ness Theorem}, {\tt arXiv:1811.07195v1} (17 Nov 2018) is a special case of
a more general statement, which can be
deduced, using a short argument, from
the classical
%%well-known
Ri\-char\-dson and Luna theorems.
%% by a short argument.
\end{abstract}

\vskip 3mm

{\bf 1.} In the recent preprint \cite{W}, the following main result
is obtained using the Kemf--Ness theorem to reduce it
to the principal orbit type theorem for compact
Lie groups:

``Let $G$ be a reductive, affine algebraic
group and let $(\rho, V )$ be a regular representation of $G$. Let $X$ be an irreducible ${\mathbb C}^\times×G$ invariant Zariski closed subset such that $G$ has a closed
orbit that has maximal dimension among all orbits (this is equivalent
to: generic orbits are closed). Then there exists an open subset, $W$, of
$X$ in the metric topology which is dense with complement of measure
$0$ such that if $x, y \in W$ then $({\mathbb C}^\times×G)_x$ is conjugate to
$({\mathbb C}^\times×G)_y$.\;Fur\-ther\-more, if $Gx$ is a closed orbit of
 maximal dimension and if $x$ is a smooth
point of $X$ then there exists $y \in W$ such that $({\mathbb C}^\times×G)_x$
contains a conjugate of $({\mathbb C}^\times×G)_y$.''

Below is shown that the more general statements can be deduced, using a short argument,  from
the classical Ri\-char\-dson and Luna theorems.
%%\cite[Prop.\,5.3]{R}, \cite[Cor.\,8 and  Rem.4 ◦ on p.\,98]{L}
%%by a short argument.

\vskip 2mm

{\bf 2.} We fix an algebraically closed ground field $k$ of characteristic $0$ and use freely the standard notation of \cite{B}, \cite{PV}.

Let $G$ be a reductive algebraic group such that  $G=CR$, where $C$ is a diagonalizable algebraic subgroup of  the center of $G$ and $R$ is a reductive algebra\-ic subgroup of $G$.\;We denote by ${\mathcal X}(C)$ the character group of $C$ and, given an algebraic $C$-module $M$ and a character $\alpha\in {\mathcal X}(C)$, by $M_\alpha$ the weight space of $M$ of the weight $\alpha$.\;Since $C$ is diagonalizable, $M$ is the direct sum of the $M_\alpha$'s; see \cite[III.8.17]{B}.

Let $X$ be irreducible affine algebraic variety endowed with a regular (mor\-phic) action of $G$.

\begin{thmnonumber}
In the above notation, assume that there  is
%%$R$ has
a closed $R$-orbit
%%that has
of  maximal dimension among all $R$-orbits in $X$.\;Then the following hold:
\begin{enumerate}[\hskip 4.2mm\rm(a)]
\item
There exists a dense open {\rm(}in the Zariski topology\,{\rm)} subset $U$ of $X$ such that if $x, y\in U$, then $G_x$ is conjugate to $G_y$.
\item %%.\;Furthermore,
If the $R$-orbit $R(z)$ of a point
%%(x)$ is a closed orbit of maximal dimension and if $x$ is a smooth point of
$z\in X$ is closed, then there exists a point $y\in U$ such that $G_z$ contains a conjugate of $G_y$.
\end{enumerate}
\end{thmnonumber}

\begin{proof}
(a) Let $S$ be the singular locus of $X$.\;We may (and shall) assume that $S\neq \varnothing$, because otherwise
the claim to be proved immediately follows from the Richardson theorem \cite[Prop.\,5.3]{R}
(see also \cite[Cor.\,8]{L}).\;As $S$ is a closed $G$-stable subset of $X$, we have $k[S]=
\bigoplus_{\alpha\in{\mathcal X}(C)}(k[S])_\alpha$.
The assumption on
%%the existence of a closed
$R$-orbit implies the existence
%%Since $X$ contains a closed $R$-orbit of maximal dimension, there is
of a dense open subset of $X$ whose points have closed $R$-orbits of maximal dimension  \cite[Thm.\;4]{P}.\;Hence there is a closed $R$-orbit $\mathcal O$ such that $S\cap \mathcal O=\varnothing$.\;This implies the existence of a function $f\in k[X]^R$ such that $f|_S=0$, $f|_{\mathcal O}=1$ (see, e.g., \cite[Lem.\;8.19(ii)]{B} or \cite[Thm.\,4.7]{PV}).
Since $C$ centralizes $R$, the algebra $k[X]^R$ is $C$-stable, so we have the weight decomposition
$k[X]^R=\bigoplus_{\alpha\in{\mathcal X}(C)}(k[X]^R)_\alpha$.\;Let $\pi_{\alpha}\colon k[X]^R\to (k[X]^R)_\alpha$ be the natural projection. Since $f|_{\mathcal O}\neq 0$, there is
$\alpha\in {\mathcal X}(C)$ such that for $f_\alpha:=\pi_{\alpha}(f)$, we have $f_{\alpha}|_{\mathcal O}\neq 0$. Since $G=CR$, the function $f_\alpha$ is a semi-invariant of $G$.\;As $S$ is $G$-stable, the homomorphism
$\varrho\colon k[X]\to k[S]$, $h\mapsto h|_S$, is $G$-equivariant, hence $\varrho(k[X]_\alpha)\subseteq k[S]_\alpha$.\;In view of $\varrho(f)=0$, this implies $\varrho(f_\alpha)=0$.\;Thus $f_\alpha$ is a nonzero semi-invariant of $G$ vanishing on $S$.\;Hence $X_{f_\alpha}:=\{x\in X\mid f_\alpha(x)\neq 0\}$ is
a $G$-stable dense open subset of $X$, which is a smooth affine variety.\;Now, by the Richardson theorem \cite[Prop.\,5.3]{R} (see also \cite[Cor.\,8]{L}),  there is a dense open subset $U$ of $X_{f_\alpha}$ such that  if $x, y\in U$, then $G_x$ is conjugate to $G_y$. This proves (a).

(b) Let $\overline{G(z)}$ be the closure of the $G$-orbit of $z$ in $X$.\;Then $B:=\overline{G(z)}\setminus G(z)$ is a closed $G$-stable subset of $X$.\;If $B=\varnothing$, then the existence of $y$ immediately follows from the Luna slice theorem, see \cite[Rem.\,$4^\circ$ on p.\,98]{L} (cf.\;\cite[Thm.\;6.3]{PV}).\;Now consider the case $B\neq \varnothing$.\;Since  $B\cap R(z)=\varnothing$, the same argument as in the above proof of (a)
shows the existence of a $G$-semi-invariant $f\!\in\! k[X]^R$ such that $f|_B=0$, $f|_{R(z)}=1$.\;The latter equality implies that $f$ vanishes nowhere on $G(z)$.\;Therefore,
$X_f$ is a $G$-stable dense open subset of $X$ containing $G(z)$, and $G(z)$ is closed in $X_f$.
Now, since $X_f$ is affine, the existence of $y$ follows from the Luna slice theorem, as above.\;This proves (b).
\end{proof}
\vskip -1mm
%%\noindent
{\bf 3.}\;{\it Remark.}\;In \cite[Sect.\;6]{W} is given an example of a linear action of a semi\-simp\-le group, which
shows that the existence of a point with trivial stabilizer does not imply triviality of
stabilizers of points in general position.\;It should be noted that
this phenomenon is not new, similar examples
%%of this phenomenon
has long been known (perhaps the earliest one belongs to Richardson, see \cite[Rem.\,$4^\circ$ on p.\,98]{L}).

%%{\bf 3.}

\end{document}